\documentclass[12pt]{amsart}
\usepackage{amsmath,amssymb,latexsym,esint,cite,mathrsfs}
\usepackage{verbatim,wasysym}
\usepackage[left=2.6cm,right=2.6cm,top=3cm,bottom=3cm]{geometry}
\makeatletter \@addtoreset{equation}{section} \makeatother

\usepackage{graphicx}
\usepackage{subfigure}
\usepackage{graphicx,epic,eepic}

\usepackage{color,enumitem,graphicx}
\usepackage[colorlinks=true,urlcolor=blue,
citecolor=red,linkcolor=blue,linktocpage,pdfpagelabels,
bookmarksnumbered,bookmarksopen]{hyperref}

\numberwithin{equation}{section}
\newtheorem{theorem}{Theorem}[section]
\newtheorem{lemma}[theorem]{Lemma}
\newtheorem{definition}[theorem]{Definition}
\newtheorem{proposition}[theorem]{Proposition}

\numberwithin{equation}{section}

\allowdisplaybreaks

\begin{document}

\title[Nonlocal Sobolev inequality]
{Remainder terms of a nonlocal Sobolev inequality}

\author[S. Deng]{Shengbing Deng}
\address{\noindent Shengbing Deng  \newline
School of Mathematics and Statistics, Southwest University,
Chongqing 400715, People's Republic of China}\email{shbdeng@swu.edu.cn}

\author[X. Tian]{Xingliang Tian}
\address{\noindent Xingliang Tian  \newline
School of Mathematics and Statistics, Southwest University,
Chongqing 400715, People's Republic of China.}\email{xltian@email.swu.edu.cn}

\author[M. Yang]{Minbo Yang}
\address{\noindent Minbo Yang  \newline
Department of Mathematics, Zhejiang Normal University,
Jinhua 321004, Zhejiang, People's Republic of China.}\email{mbyang@zjnu.edu.cn}

\author[S. Zhao]{Shunneng Zhao}
\address{\noindent Shunneng Zhao  \newline
Department of Mathematics, Yunnan Normal University,
Kunming 650500, Yunnan, People's Republic of China.}\email{snzhao@zjnu.edu.cn}



\allowdisplaybreaks

\begin{abstract}
{\tiny In this note we study a nonlocal version of the Sobolev inequality
\begin{equation*}
\int_{\mathbb{R}^N}|\nabla u|^2 dx \geq S_{HLS}\left(\int_{\mathbb{R}^N}\big(|x|^{-\alpha} \ast u^{2_\alpha^{\ast}}\big)u^{2_\alpha^{\ast}} dx\right)^{\frac{1}{2_\alpha^{\ast}}}, \quad \forall u\in \mathcal{D}^{1,2}(\mathbb{R}^N),
\end{equation*}
where $S_{HLS}$ is the best constant, $\ast$ denotes the standard convolution and $\mathcal{D}^{1,2}(\mathbb{R}^N)$ denotes the classical Sobolev space with respect to the norm $\|u\|_{\mathcal{D}^{1,2}(\mathbb{R}^N)}=\|\nabla u\|_{L^2(\mathbb{R}^N)}$. By using the nondegeneracy property of the extremal functions, we prove that the existence of the gradient type remainder term and a reminder term in the weak $L^{\frac{N}{N-2}}$-norm of above inequality for all $0<\alpha<N$.
    }
\end{abstract}

\vspace{3mm}

\maketitle

\section{Introduction}\label{sectid}

\subsection{Motivation}
Recall that the classical Sobolev inequality: for $N\geq 3$ there exists $\mathcal{S}=\mathcal{S}(N)>0$ such that
\begin{equation}\label{bsic}
\|\nabla u\|^2_{L^2(\mathbb{R}^N)}\geq \mathcal{S}\|u\|^2_{L^{2^*}(\mathbb{R}^N)},\quad  \forall u\in \mathcal{D}^{1,2}(\mathbb{R}^N),
\end{equation}
where $2^*=2N/(N-2)$ and $\mathcal{D}^{1,2}(\mathbb{R}^N)$ denotes the closure of $C^\infty_c(\mathbb{R}^N)$ with respect to the norm $\|u\|_{\mathcal{D}^{1,2}(\mathbb{R}^N)}=\|\nabla u\|_{L^2(\mathbb{R}^N)}$. It is well known that the Euler-Lagrange equation associated to (\ref{bsic}) is
\begin{equation}\label{bec}
-\Delta u=|u|^{2^*-2}u\quad \mbox{in}\ \mathbb{R}^N.
\end{equation}
By Caffarelli et al. \cite{CGS89} and Gidas et al. \cite{GNN79}, it is known that all positive solutions of equation (\ref{bec}) are Talenti functions \cite{Ta76} $V_{z,\lambda}(x)=\lambda^\frac{N-2}{2}V(\lambda(x-z))$ with $z\in\mathbb{R}^N$ and $\lambda>0$, where \[V(x)=[N(N-2)]^{\frac{N-2}{4}}({1+|x|^2})^{-\frac{N-2}{2}}.\]
In a celebrated paper \cite{BrL85}, Br\'{e}zis and Lieb asked the question whether a remainder term proportional to the quadratic distance of the function $u$ to be the manifold $\mathcal{M}_0:=\{cV_{z,\lambda}: z\in\mathbb{R}^N, c\in\mathbb{R}, \lambda>0\}$ - can be added to the right hand side of (\ref{bsic}).  This question was answered affirmatively by Bianchi and Egnell \cite{BE91}. To study the remainder term for the Sobolev inequality (\ref{bsic}), Bianchi and Egnell \cite{BE91} (see also \cite[Lemma 3.1]{AGP99}) established firstly the non-degeneracy of  $V$ was given, that is, solutions of the linearized equation
\begin{equation}
-\Delta v=(2^*-1)V^{2^*-2}v\quad \mbox{in}\ \mathbb{R}^N,\quad v\in \mathcal{D}^{1,2}(\mathbb{R}^N),
\end{equation}
are linear combinations of functions $\frac{N-2}{2}V+x\cdot\nabla V$ and $\partial_{x_i} V$, $i=1,\ldots,N$.
  It is worth mentioning that the non-degeneracy property of solutions has many significant applications in studying the concentration behavior or the blow-up phenomena of the solutions.  The remainder term results established by Bianchi and Egnell \cite{BE91} were  extended later to the biharmonic case \cite{LW00}, the polyharmonic case \cite{BWW03} and the fractional-order case \cite{CFW13}. In addition, Figalli and Neumayer \cite{FN19}, Figalli and Zhang \cite{FZ22} obtained the remainder terms of $p$-Laplace Sobolev inequality. Recently, Ciraolo et al. \cite{CFM18}, Figalli and Glaudo \cite{FG20}, Deng et al. \cite{DSW21} obtained sharp quantitative estimates of Struwe's decomposition (see \cite{St84}) for Sobolev inequality (\ref{bsic}) by using the finite-dimensional reduction method. It is worth to mention that Wang and Willem \cite{WaWi03} obtained the remainder terms of Caffarelli-Kohn-Nirenberg inequalities (see \cite{CKN84}), Wei and Wu \cite{WW22} established the stability of the profile decompositions to the Caffarelli-Kohn-Nirenberg inequalities and also gave the gradient type remainder terms.

To state the motivation of the present problem, we also need to recall the well-known Hardy-Littlewood-Sobolev (HLS for short) inequality, see \cite{Lieb83} and also \cite{DE22} for recent results about HLS inequality and the references therein.

    \begin{proposition}\label{prohlsi}
    ({\bfseries HLS inequality}) Let $r,t>1$ and $0<\alpha<N$ with $\frac{1}{r}+\frac{1}{t}+\frac{\alpha}{N}=2$, $u\in L^r(\mathbb{R}^N)$ and $v\in L^t(\mathbb{R}^N)$. There exists a sharp constant $C(N,r,t,\alpha)>0$, independent of $u$ and $v$, such that
    \begin{equation}\label{hlsi}
    \int_{\mathbb{R}^N}\int_{\mathbb{R}^N}\frac{u(x)v(y)}{|x-y|^\alpha}dydx\leq C(N,r,t,\alpha)\|u\|_{L^r(\mathbb{R}^N)}\|v\|_{L^t(\mathbb{R}^N)}.
    \end{equation}
    If $t=r=\frac{2N}{2N-\alpha}$, then
    \begin{equation}\label{defhlsbc}
    C(N,r,t,\alpha)=C(N,\alpha)=\frac{\Gamma((N-\alpha)/2)\pi^{\alpha/2}}{\Gamma(N-\alpha/2)}\left(\frac{\Gamma(N)}{\Gamma(N/2)}\right)^{(N-\alpha)/N},
    \end{equation}
    and there is equality in (\ref{hlsi}) if and only if $u\equiv (const.) v$ and
    \begin{equation*}
    v(x)=A(1+\lambda^2|x-z|^2)^{-(2N-\alpha)/2}
    \end{equation*}
    for some $A\in \mathbb{C}$, $\lambda\in \mathbb{R}\backslash\{0\}$ and $z\in \mathbb{R}^N$.
    \end{proposition}

    Therefore, it is natural to ask whether a remainder term can be added to the left hand side of (\ref{hlsi}) at least for the special case  $t=r=\frac{2N}{2N-\alpha}$. Successfully, Carlen \cite{Ca17} gave the affirmative answer, that is

\vskip0.25cm

    \noindent{\bf Theorem~A.} \cite[Theorem 1.5]{Ca17} {\it Let $N\geq 2$ and $0<\alpha<N$. There is a constant $\kappa_{HLS,\alpha}>0$ depending only on $N$ and $\alpha$ such that for all $u\in L^{\frac{2N}{2N-\alpha}}(\mathbb{R}^N)$, it holds that
    \begin{equation*}
    C(N,\alpha)\|u\|^2_{L^{\frac{2N}{2N-\alpha}}(\mathbb{R}^N)}-\int_{\mathbb{R}^N}\int_{\mathbb{R}^N}\frac{u(x)u(y)}{|x-y|^\alpha}dydx
    \geq \kappa_{HLS,\alpha}\inf_{v\in \mathcal{M}_{HLS,\alpha}}\|u-v\|^2_{L^{\frac{2N}{2N-\alpha}}(\mathbb{R}^N)}.
    \end{equation*}
    Here $C(N,\alpha)$ is the best constant as given in (\ref{defhlsbc}), and $\mathcal{M}_{HLS,\alpha}$ is the manifold of all  optimal functions, which is generated from $h(x)=(1+|x|^2)^{-(2N-\alpha)/2}$ by multiplication by a constant, translations and scalings.
    }

\vskip0.25cm

    In the present paper, we will give some other type remainder terms of HLS inequality for $0<\alpha<N$ which are different from Carlen's \cite{Ca17}.
    By the HLS inequality, we know
    \begin{equation*}
    \int_{\mathbb{R}^N}\int_{\mathbb{R}^N}\frac{|u(x)|^p|v(y)|^p}{|x-y|^\alpha}dydx
    \end{equation*}
    is well-defined if for $u\in L^{pt}(\mathbb{R}^N)$ satisfying
    \[\frac{2}{t}+\frac{\alpha}{N}=2.\]
     Therefore, for $u\in H^1(\mathbb{R}^N)$ where 
    $H^1(\mathbb{R}^N)=\{u\in L^2(\mathbb{R}^N)|\partial_{x_i}u\in L^2(\mathbb{R}^N), i=1,\ldots,N\}$, the Sobolev continuous embedding
    $H^1(\mathbb{R}^N)\hookrightarrow L^q(\mathbb{R}^N)$ for every $q\in \left[2,\frac{2N}{N-2}\right]$
    implies
    \begin{equation*}
    \frac{2N-\alpha}{N}\leq p\leq \frac{2N-\alpha}{N-2}.
    \end{equation*}
    Thus we may call $\frac{2N-\alpha}{N}$ the low critical exponent and $\frac{2N-\alpha}{N-2}$ the upper critical exponent due to the HLS inequality.

\subsection{Problem setup and main results}
Combining the Hardy-Littlewood-Sobolev inequality and the Sobolev inequality,   we are ging to consider the following nonlocal Sobolev inequality
\begin{equation}\label{Prm}
\int_{\mathbb{R}^N}|\nabla u|^2 dx
\geq S_{HLS}\left(\int_{\mathbb{R}^N}(|x|^{-\alpha} \ast|u|^{2^*_\alpha})|u|^{2^*_\alpha} dx\right)^{\frac{1}{2^*_\alpha}}, \quad \forall u\in \mathcal{D}^{1,2}(\mathbb{R}^N),
\end{equation}
for some positive constant $S_{HLS}$ depending only on $N$ and $\alpha$, where $N\geq 3$, $0<\alpha<N$, $2^*_{\alpha}=\frac{2N-\alpha}{N-2}$ and $\ast$ denotes convolution which means
 \[(f\ast g)(x)=\int_{\mathbb{R}^N}f(x-y)g(y)dy.\]

 The minimizing problem for the best constant $S_{HLS}$ in (\ref{Prm}) can be defined as
 \begin{equation*}
S_{HLS}= \inf_{u\in \mathcal{D}^{1,2}(\mathbb{R}^N)\backslash\{0\}}\frac{\int_{\mathbb{R}^N}|\nabla u|^2 dx}{\left(\int_{\mathbb{R}^N}(|x|^{-\alpha} \ast|u|^{2^*_\alpha})|u|^{2^*_\alpha} dx\right)^{\frac{1}{2^*_\alpha}}}.
\end{equation*}
Gao and Yang \cite{GY18} proved the best constant $S_{HLS}$ in (\ref{Prm}) satisfies
    \begin{equation*}
    S_{HLS}=\frac{\mathcal{S}}{[C(N,\alpha)]^{\frac{N-2}{2N-\alpha}}},
    \end{equation*}
    where $\mathcal{S}$ is the best Sobolev constant and $C(N,\alpha)$ is the sharp constant of HLS inequality given as in (\ref{defhlsbc}).  What's more, $S_{HLS}$ is achieved if and only if by
    \begin{equation}\label{defaf}
    cU_{\lambda,z,\alpha}(x)=c\lambda^{\frac{N-2}{2}}U_\alpha(\lambda (x-z))
    \end{equation}
    for some $c\in\mathbb{R}\backslash\{0\}$, $\lambda>0$ and $z\in \mathbb{R}^N$, where
    \begin{equation}\label{defU}
    U_\alpha(x)=\mathcal{S}^{\frac{(N-\alpha)(2-N)}{4(N-\alpha+2)}}[C(N,\alpha)]^{\frac{2-N}{2(N-\alpha+2)}}\frac{[N(N-2)]^{\frac{N-2}{4}}}{(1+|x|^2)^{\frac{N-2}{2}}},
    \end{equation}
    which satisfies the the Euler-Lagrange equation of (\ref{Prm}), that is
    \begin{equation}\label{ele}
    -\Delta u=(|x|^{-\alpha}\ast |u|^{2^*_\alpha})|u|^{2^*_\alpha-2}u \quad \mbox{in}\quad \mathbb{R}^N.
    \end{equation}
    See also \cite{AGSY17,DY19,GHPS19,Le18}. However, for every open set $\Omega$ of $\mathbb{R}^N$ for $N\geq 3$,
    \begin{equation*}
    	S_{HLS}(\Omega):= \inf_{u\in \mathcal{D}^{1,2}(\Omega)\backslash\{0\}}\frac{\int_{\Omega}|\nabla u|^2 dx}{\left(\int_{\Omega}\int_{\Omega}\frac{|u(y)|^{2^*_\alpha}|u(x)|^{2^*_\alpha}}{|x-y|^\alpha} dydx\right)^{\frac{1}{2^*_\alpha}}},
    \end{equation*}
    then $S_{HLS}(\Omega)=S_{HLS}$ while $S_{HLS}(\Omega)$ is never achieved except when $\Omega=\mathbb{R}^N$. Furthermore, Giacomoni, Wei and Yang \cite{GWY20} proved that if $\alpha$ is close to $0$, $U_\alpha(x)$ as in (\ref{defU}) is nondegenerate in the sense that solutions of the linearized equation
    \begin{equation}\label{deflp}
    -\Delta v=2^*_\alpha\left(|x|^{-\alpha} \ast (U_\alpha^{2^*_\alpha-1}v)\right)U_\alpha^{2^*_\alpha-1}
    + (2^*_\alpha-1)\left(|x|^{-\alpha} \ast U_\alpha^{2^*_\alpha}\right)U_\alpha^{2^*_\alpha-2}v\quad \mbox{in}\quad \mathbb{R}^N,
    \end{equation}
    $v\in \mathcal{D}^{1,2}(\mathbb{R}^N)$ are linear combinations of functions $\frac{N-2}{2}U_\alpha+x\cdot\nabla U_\alpha$ and $\partial_{x_i} U_\alpha$, $i=1,\ldots,N$.
    Du and Yang \cite{DY19} also proved that if $\alpha$ is close to $N$ with $N=3$ or $4$, $U_\alpha(x)$ as in (\ref{defU}) is nondegenerate. Furthermore, Gao et al \cite{GMYZ22} proved that when $N=6$, $U_\alpha(x)$ as in (\ref{defU}) with $\alpha=4$ is nondegenerate (see also \cite{XTang} for the case $\alpha=4$) and confirms an open
nondegeneracy problem in \cite{GMYZ22}. More recently, Li et al. \cite{XLi} applied the  spherical harmonic decomposition
and the Funk-Hecke formula of the spherical harmonic functions to obtain the
nondegeneracy of positive bubble solutions for critical Hartree equation (\ref{ele}).
More precisely, We summarize those results as follows:

    \vskip0.25cm

    \begin{proposition}\label{prondgr}
    Assume $N\geq 3$, $0<\alpha<N$ and let $U_{\lambda,z,\alpha}$ be as in (\ref{defaf}).
    Then the linearized operator of equation (\ref{ele}) at $U_{\lambda,z,\alpha}$ defined by
    \begin{equation*}
    \mathcal{L}(v)=-\Delta v-2^*_\alpha\left(|x|^{-\alpha} \ast (U_{\lambda,z,\alpha}^{2^*_\alpha-1}v)\right)U_{\lambda,z,\alpha}^{2^*_\alpha-1}
    - (2^*_\alpha-1)\left(|x|^{-\alpha} \ast U_{\lambda,z,\alpha}^{2^*_\alpha}\right)U_{\lambda,z,\alpha}^{2^*_\alpha-2}v
    \end{equation*}
    only admits solutions in $\mathcal{D}^{1,2}(\mathbb{R}^N)$ of the form
    \begin{equation*}
    v=\overline{a}D_\lambda U_{\lambda,z,\alpha} +\mathbf{b}\cdot\nabla U_{\lambda,z,\alpha},
    \end{equation*}
    where $\overline{a}\in\mathbb{R}$, $\mathbf{b}\in\mathbb{R}^N$.
    \end{proposition}


    In the present paper, we first are concerned with the remainder terms of the nonlocal Sobolev inequality
    \begin{equation}\label{Prm64}
    \int_{\mathbb{R}^N}|\nabla u|^2 dx
    \geq S_{HLS}\left(\int_{\mathbb{R}^N}(|x|^{-\alpha} \ast u^{2_{\alpha}^{\ast}})u^{2_{\alpha}^{\ast}} dx\right)^{\frac{1}{2_{\alpha}^{\ast}}}, \quad \forall u\in \mathcal{D}^{1,2}(\mathbb{R}^N),
    \end{equation}
 for $0<\alpha<N$.

 We prove that the gradient type remainder term of inequality (\ref{Prm64}) stated as
    \begin{theorem}\label{thmprt}
    Assume that $N\geq3$ and $0<\alpha<N$.
    There exist two constants $B_2\geq B_1$ such that for every $u\in \mathcal{D}^{1,2}(\mathbb{R}^N)$, it holds that
    \[
    B_2{\rm dist}(u,\mathcal{M})^2
    \geq \int_{\mathbb{R}^N}|\nabla u|^2 dx
    -S_{HLS}\left(\int_{\mathbb{R}^N}\big(|x|^{-\alpha} \ast u^{2_{\alpha}^{\ast}}\big)u^{2_{\alpha}^{\ast}} dx\right)^{\frac{1}{2_{\alpha}^{\ast}}}
    \geq B_1 {\rm dist}(u,\mathcal{M})^2,
    \]
    where $\mathcal{M}=\{cU_{\lambda,z}: c\in\mathbb{R}, \lambda>0, z\in\mathbb{R}^N\}$ is $N+2$-dimensional manifold (see $U_{\lambda,z}$ in (\ref{defaf})), and ${\rm dist}(u,\mathcal{M }):=\inf_{c\in\mathbb{R}, \lambda>0, z\in\mathbb{R}^N}\|u-cU_{\lambda,z}\|_{\mathcal{D}^{1,2}(\mathbb{R}^N)}$.
    \end{theorem}

    Furthermore,  we will also consider the problem set in bounded domain. In sprite of the work of Br\'{e}zis and Lieb \cite{BrL85}, for each bounded domain $\Omega\subset\mathbb{R}^N$ with $N\geq 3$, we define the weak $L^q$-norm
\begin{equation}\label{defwn}
\|u\|_{L^q_w(\Omega)}:=\sup_{D\subset\Omega}\frac{\int_{D}|u|dx}{|D|^{\frac{q-1}{q}}}.
\end{equation}
It is easy to verify that $\|u\|_{L^q_w(\Omega)}\leq \|u\|_{L^q(\Omega)}$. Br\'{e}zis and Nirenberg \cite{BN83} proved that if $1\leq q<\frac{N}{N-2}$ then there exists $A_q=A(N,q,\Omega)>0$ such that
\begin{equation*}
\|\nabla u\|^2_{L^2(\Omega)}-\mathcal{S}\|u\|^2_{L^{2^*}(\Omega)}\geq A_q\|u\|^2_{L^q(\Omega)},\quad  \forall u\in \mathcal{D}^{1,2}_0(\Omega),
\end{equation*}
where $\mathcal{S}$ is the best Sobolev constant and $\mathcal{D}^{1,2}_0(\Omega):=\{u\in L^2(\Omega)|\partial_{x_i}u\in L^2(\Omega), i=1,\ldots, N\}$, with $A_q\to 0$ as $q\to \frac{N}{N-2}$. Furthermore, the result is sharp in the sense that is not true if $q=\frac{N}{N-2}$. However, the following refinement is proved by Br\'{e}zis and Lieb \cite{BrL85}:
\begin{equation*}
\|\nabla u\|^2_{L^2(\Omega)}-\mathcal{S}\|u\|^2_{L^{2^*}(\Omega)}\geq A'\|u\|^2_{L^{\frac{N}{N-2}}_w(\Omega)},\quad  \forall u\in \mathcal{D}^{1,2}_0(\Omega),
\end{equation*}
for some constant $A'>0$, here $\|\cdot\|_{L^q_w(\Omega)}$ is the weak $L^q$-norm defined as in (\ref{defwn}). In addition, Bianchi and Egnell \cite{BE91} gave a new and simple proof of previous inequality by using the nondegeneracy property of extremal functions.

Then we give the second type remainder term of inequality (\ref{Prm64}) in bounded domain stated as
\begin{theorem}\label{thmprtb}
    Let $\Omega\subset \mathbb{R}^N$ be a bounded domain, $N\geq3$ and $0<\alpha<N$. Then there exists constant $B'>0$ such that for every $u\in \mathcal{D}^{1,2}_0(\Omega)$, it holds that
    \begin{equation}\label{rtbdy46}
    \int_{\Omega}|\nabla u|^2 dx
    -S_{HLS}\left(\int_{\Omega}\int_{\Omega}\frac{u^{2_{\alpha}^{\ast}}(x)u^{2_{\alpha}^{\ast}}(y)}{|x-y|^\alpha} dxdy\right)^{\frac{1}{2_{\alpha}^{\ast}}}
    \geq B' \big\|u\big\|^2_{L^{\frac{N}{N-2}}_w(\Omega)},
    \end{equation}
    where $L^{\frac{N}{N-2}}_w$ denotes the weak $L^{\frac{N}{N-2}}$-norm as in (\ref{defwn}).
    Moreover, the weak norm on the right-hand side cannot be replaced by the strong norm.
    \end{theorem}

    The paper is organized as follows. In Section \ref{sectevp}, we first give the definitions of eigenvalues for a linear problem, then from the nondegeneracy of $U_\alpha$ we obtain the eigenvalues and corresponding eigenfunctions. Then in Section \ref{sectrt} we give the behavior near minimizers manifold $\mathcal{M}$ and complete the proof of Theorem \ref{thmprt}. Finally, by using the conclusion of Theorem \ref{thmprt} and rearrangement theory, we give the proof of Theorem \ref{thmprtb}.

\section{An eigenvalue problem}\label{sectevp}

    For simplicity of notations, we write $U_{\lambda,z}$ instead of $U_{\lambda,z,\alpha}$ and $U$ instead of $U_{\alpha}$ as in (\ref{defaf}) and (\ref{defU}) respectively if there is no possibility of confusion.

Let us consider the following eigenvalue problem
    \begin{equation}\label{Pwhlep}
    -\Delta v+\big(|x|^{-\alpha}\ast U^{2_{\alpha}^{\ast}}\big) U^{2_{\alpha}^{\ast}-2}v
    =\mu\big[\big(|x|^{-\alpha} \ast (U^{2_{\alpha}^{\ast}-1}v)\big)U^{2_{\alpha}^{\ast}-1}
    +\big(|x|^{-\alpha} \ast U^{2_{\alpha}^{\ast}}\big)U^{2_{\alpha}^{\ast}-2}v\big],\quad v\in \mathcal{D}^{1,2}(\mathbb{R}^N).
    \end{equation}
    By the HLS inequality, H\"{o}lder inequality and Sobolev inequality, we can obtain for all $v\in \mathcal{D}^{1,2}(\mathbb{R}^N)$, it holds that
    \begin{equation}\label{pfvlnv}
    \begin{split}
    \int_{\mathbb{R}^N}\Big(|x|^{-\alpha} \ast (U^{2_{\alpha}^{\ast}-1}v)\Big)U^{2_{\alpha}^{\ast}-1}v dx
    \leq & C(N,\alpha) \|U^{2_{\alpha}^{\ast}-1}v\|^2_{L^{\frac{2N}{2N-\alpha}}(\mathbb{R}^N)}  \\
    \leq & C(N,\alpha) \big\|U\big\|^{2(2_{\alpha}^{\ast}-1)}_{L^{2^{\ast}}(\mathbb{R}^N)}\big\|v\big\|^{2}_{L^{2^{\ast}}(\mathbb{R}^N)} \\
    = &  \mathcal{S} \big\|v\big\|^{2}_{L^{2^{\ast}}(\mathbb{R}^N)}
    \leq \big\|v\big\|^2_{\mathcal{D}^{1,2}(\mathbb{R}^N)},
    \end{split}
    \end{equation}
    where $C(N,\alpha)$ is the sharp constant of HLS inequality given as in (\ref{defhlsbc}) and $\mathcal{S}$ is the best Sobolev constant. Similarly, we have
    \begin{equation}\label{pfvlnv2}
    \begin{split}
    \int_{\mathbb{R}^N}(|x|^{-\alpha} \ast U^{2_{\alpha}^{\ast}})U^{2_{\alpha}^{\ast}-2}v^2 dx
    \leq \big\|v\big\|^2_{\mathcal{D}^{1,2}(\mathbb{R}^N)}.
    \end{split}
    \end{equation}
    Then following the work of Servadei and Valdinoci \cite{SV13} for the nonlocal case, we can introduce the definitions of eigenvalues of problem (\ref{Pwhlep}) as the following

    \begin{definition}\label{defevp}
    The first eigenvalue of problem (\ref{Pwhlep}) can be defined as
    \begin{equation}\label{deffev1}
    \mu_1:=\inf_{v\in \mathcal{D}^{1,2}(\mathbb{R}^N)\backslash\{0\}}
    \frac{\int_{\mathbb{R}^N}|\nabla v|^2 dx+\int_{\mathbb{R}^N}(|x|^{-\alpha} \ast U^{2_{\alpha}^{\ast}})U^{2_{\alpha}^{\ast}-2}v^2 dx}
    {\int_{\mathbb{R}^N}(|x|^{-\alpha} \ast (U^{2_{\alpha}^{\ast}-1}v))U^{2_{\alpha}^{\ast}-1}v dx+ \int_{\mathbb{R}^N}(|x|^{-\alpha} \ast U^{2_{\alpha}^{\ast}})U^{2_{\alpha}^{\ast}-2}v^2 dx}
    .
    \end{equation}
    Moreover, for any $k\in\mathbb{N}^+$ the (k+1)th-eigenvalues can be characterized as follows:
    \begin{equation}\label{deffevk}
    \mu_{k+1}:=\inf_{v\in \mathbb{P}_{k+1}\backslash\{0\}}
    \frac{\int_{\mathbb{R}^N}|\nabla v|^2 dx+\int_{\mathbb{R}^N}(|x|^{-\alpha} \ast U^{2_{\alpha}^{\ast}}) U^{2_{\alpha}^{\ast}-2}v^2 dx}
    {\int_{\mathbb{R}^N}(|x|^{-\alpha} \ast (U^{2_{\alpha}^{\ast}-1}v))U^{2_{\alpha}^{\ast}-1}v dx+ \int_{\mathbb{R}^N}(|x|^{-\alpha} \ast U^{2_{\alpha}^{\ast}})U^{2_{\alpha}^{\ast}-2}v^2 dx}
    ,
    \end{equation}
    where
    \begin{equation}\label{defczs}
    \mathbb{P}_{k+1}:=\left\{v\in \mathcal{D}^{1,2}(\mathbb{R}^N): \int_{\mathbb{R}^N}\nabla v \cdot \nabla e_j dx=0,\quad \mbox{for all}\quad j=1,\ldots,k\right\},
    \end{equation}
    and $e_j$ is the corresponding eigenfunction to $\mu_j$.
    \end{definition}

    Choosing $v=U$ in (\ref{deffev1}), since $U$ is the solution of equation (\ref{ele}) then we have
    \begin{equation*}
    \begin{split}
    \mu_1
    \leq &  \frac{\int_{\mathbb{R}^N}|\nabla U|^2 dx +\int_{\mathbb{R}^N}(|x|^{-\alpha} \ast U^{2_{\alpha}^{\ast}})U^{2_{\alpha}^{\ast}} dx}
    {\int_{\mathbb{R}^N}(|x|^{-\alpha} \ast U^{2_{\alpha}^{\ast}})U^{2_{\alpha}^{\ast}} dx
    + \int_{\mathbb{R}^N}(|x|^{-\alpha} \ast U^{2_{\alpha}^{\ast}})U^{2_{\alpha}^{\ast}} dx}=1.
    \end{split}
    \end{equation*}
    Moreover, from (\ref{pfvlnv2}) we can deduce that for all $v\in \mathcal{D}^{1,2}(\mathbb{R}^N)$,
    \begin{equation*}
    \begin{split}
    \int_{\mathbb{R}^N}(|x|^{-\alpha} \ast (U^{2_{\alpha}^{\ast}-1}v))U^{2_{\alpha}^{\ast}-1}v dx
    \leq \big\|v\big\|^2_{\mathcal{D}^{1,2}(\mathbb{R}^N)},
    \end{split}
    \end{equation*}
    and equality holds if and only if $v=\zeta U$ with $\zeta\in\mathbb{R}$, which implies $\mu_1\geq 1$. Then we have $\mu_1=1$ and the corresponding eigenfunction is $\zeta U$ with $\zeta\in\mathbb{R}$, that is $e_1=U$.
    Then from Proposition \ref{prondgr} regarding the nondegeneracy of $U$, we have 

    \begin{proposition}\label{propep}
    Let $\mu_i$, $i=1,2,\ldots,$ denote the eigenvalues of (\ref{Pwhlep}) in increasing order as in Definition \ref{defevp}. Then $\mu_1=1$ is simple and the corresponding eigenfunction is $\zeta U$ with $\zeta\in\mathbb{R}$, and there exists $k\in\mathbb{N}$ such that $\mu_{k+2}=\mu_{k+3}=\cdots=\mu_{k+N+2}=2_{\alpha}^{\ast}$ with the corresponding $(N+1)$-dimensional eigenfunction space spanned by
    \[
    \left\{2U+x\cdot\nabla U,\quad \partial_{x_1} U,\ldots,\partial_{x_N} U\right\}.
    \]
    Furthermore, $\mu_{k+N+3}>\mu_{k+2}=2_{\alpha}^{\ast}$.
    \end{proposition}

\section{Remainder terms of HLS type inequality}\label{sectrt}

    The main ingredient of the proof of Theorem \ref{thmprt} is contained in Lemma \ref{lemma:rtnm2b} below, where the behavior of the sequences near $\mathcal{M}=\{cU_{\lambda,z}: c\in\mathbb{R}, \lambda>0, z\in\mathbb{R}^N\}$ is investigated. 

    \begin{lemma}\label{lemma:rtnm2b}
    For any sequence $\{u_n\}\subset \mathcal{D}^{1,2}(\mathbb{R}^N)\backslash \mathcal{M}$ satisfying
    \[\inf_n\|u_n\|>0,\quad {\rm dist}(u_n,\mathcal{M})\to 0,\]
    then we have
    \begin{equation}\label{rtnmb}
    \lim\inf_{n\to\infty}
    \frac{\int_{\mathbb{R}^N}|\nabla u_n|^2 dx
    -S_{HLS}\left(\int_{\mathbb{R}^N}(|x|^{-\alpha} \ast u_n^{2_{\alpha}^{\ast}})u_n^{2_{\alpha}^{\ast}} dx\right)^{\frac{1}{2_{\alpha}^{\ast}}}}
    {{\rm dist}(u_n,\mathcal{M})^2}
    \geq 2(\mu_{k+N+3}-\mu_{k+N+2}),
    \end{equation}
    and
    \begin{equation}\label{rtnml}
    \lim\sup_{n\to\infty}
    \frac{\int_{\mathbb{R}^N}|\nabla u_n|^2 dx
    -S_{HLS}\left(\int_{\mathbb{R}^N}(|x|^{-\alpha} \ast u_n^{2_{\alpha}^{\ast}})u_n^{2_{\alpha}^{\ast}} dx\right)^{\frac{1}{2_{\alpha}^{\ast}}}}
    {{\rm dist}(u_n,\mathcal{M})^2}
    \leq 1,
    \end{equation}
    where $\mu_{k+N+2}=2 <\mu_{k+N+3}$ are defined in Proposition \ref{propep}. Moreover, this conclusion implies $2 <\mu_{k+N+3}\leq 2_{\alpha}^{\ast}+\frac{1}{2}$.
    \end{lemma}

    \begin{proof}
    Let $d_n:={\rm dist}(u_n,\mathcal{M})=\inf_{c\in\mathbb{R}, \lambda>0, z\in\mathbb{R}^6}\|u_n-cU_{\lambda,z}\|_{\mathcal{D}^{1,2}(\mathbb{R}^N)}\to 0$, as $n\to \infty$. It is well known that for each $u_n\in \mathcal{D}^{1,2}(\mathbb{R}^N)$, there exists $(c_n,\lambda_n,z_n)\in \mathbb{R}\times\mathbb{R}^+\times\mathbb{R}^N$ such that
    \[d_n=\big\|u_n-c_nU_{\lambda_n,z_n}\big\|_{\mathcal{D}^{1,2}(\mathbb{R}^N)}.\]
    Since $\mathcal{M}$ is $(N+2)$-dimensional manifold embedded in $\mathcal{D}^{1,2}(\mathbb{R}^N)$, that is
    \[
    (c,\lambda,z)\in\mathbb{R}\times\mathbb{R}^+\times\mathbb{R}^N\to cU_{\lambda,z}\in \mathcal{D}^{1,2}(\mathbb{R}^N),
    \]
    then from Proposition \ref{propep} and a simple scaling argument, we have the tangential space at $(c_n,\lambda_n,z_n)$ is given by $T_{c_n U_{\lambda_n,z_n}}\mathcal{M}$ which is spanned by
    \[
    U_{\lambda_n,z_n}, \quad e_{2},\ldots,e_{k+1},\quad \frac{\partial U_{\lambda,z_n}}{\partial \lambda}\Big|_{\lambda=\lambda_n},\frac{\partial U_{\lambda_n,z}}{\partial z_1}\Big|_{z=z_n},\ldots,\frac{\partial U_{\lambda_n,z}}{\partial z_N}\Big|_{z=z_n},
    \]
    where $e_j$ is the corresponding eigenfunction to $\mu_j$ for $j=2,\ldots,k+1$ if $k\geq 1$ as defined in Proposition \ref{propep}, and if $k=0$ then
    \[
    T_{c_n U_{\lambda_n,z_n}}\mathcal{M}={\rm Span}\left\{U_{\lambda_n,z_n}, \quad \frac{\partial U_{\lambda,z_n}}{\partial \lambda}\Big|_{\lambda=\lambda_n},\frac{\partial U_{\lambda_n,z}}{\partial z_1}\Big|_{z=z_n},\ldots,\frac{\partial U_{\lambda_n,z}}{\partial z_N}\Big|_{z=z_n}\right\}.
    \]
    Anyway we must have that $(u_n-c_n U_{\lambda_n,z_n})$ is perpendicular to $T_{c_n U_{\lambda_n,z_n}}\mathcal{M}$, particularly,
    \[\int_{\mathbb{R}^N}\nabla U_{\lambda_n,z_n} \cdot\nabla (u_n-c_n U_{\lambda_n,z_n}) dx=0.\]
    Let
    \[u_n=c_n U_{\lambda_n,z_n}+d_n w_n,\]
     then $w_n$ is perpendicular to $T_{c_n U_{\lambda_n,z_n}}\mathcal{M}$, $\|w_n\|_{\mathcal{D}^{1,2}(\mathbb{R}^N)}=1$ and
    \begin{equation*}
    \|u_n\|^2_{\mathcal{D}^{1,2}(\mathbb{R}^N)}=d_n^2+c_n^2\|U_{\lambda_n,z_n}\|^2_{\mathcal{D}^{1,2}(\mathbb{R}^N)}
    =d_n^2+c_n^2\|U\|^2_{\mathcal{D}^{1,2}(\mathbb{R}^N)},
    \end{equation*}
    since
    \begin{equation*}
    \begin{split}
    \int_{\mathbb{R}^N}\nabla U_{\lambda_n,z_n} \cdot\nabla w_n dx=0\quad \mbox{and}\quad \|U_{\lambda_n,z_n}\|^2_{\mathcal{D}^{1,2}(\mathbb{R}^N)}=\|U\|^2_{\mathcal{D}^{1,2}(\mathbb{R}^N)}.
    \end{split}
    \end{equation*}
    Notice that
    \begin{equation*}
    \begin{split}
    \big|u_n\big|^{2_{\alpha}^{\ast}}=\big|c_nU_{\lambda_n,z_n}+d_nw_n\big|^{2_{\alpha}^{\ast}}
    = & |c_n|^{2_{\alpha}^{\ast}}U_{\lambda_n,z_n}^{^{2_{\alpha}^{\ast}}}
    +2_{\alpha}^{\ast}|c_n|^{2_{\alpha}^{\ast}-1} U_{\lambda_n,z_n}^{2_{\alpha}^{\ast}-1}w_nd_n\\&
    +\frac{2_{\alpha}^{\ast}(2_{\alpha}^{\ast}-1)}{2}|c_n|^{2_{\alpha}^{\ast}-2}U_{\lambda_n,z_n}^{^{2_{\alpha}^{\ast}}-2}w_n^2d_n^2+o(d_n^2),
    \end{split}
    \end{equation*}
    then we have
    \begin{equation}\label{epkeyiybb}
    \begin{split}
    \int_{\mathbb{R}^N}\Big(|x|^{-\alpha} \ast u_n^{2_{\alpha}^{\ast}}\Big)u_n^{2_{\alpha}^{\ast}} dx
    = & c_n^{2\cdot2_{\alpha}^{\ast}}\int_{\mathbb{R}^N}\Big(|x|^{-\alpha} \ast U_{\lambda_n,z_n}^{2_{\alpha}^{\ast}}\Big)U_{\lambda_n,z_n}^{2_{\alpha}^{\ast}} dx\\&
    + 2_{\alpha}^{\ast}\cdot(2_{\alpha}^{\ast}-1)|c_n|^{2_{\alpha}^{\ast}(2_{\alpha}^{\ast}-1)}d_n^2 \int_{\mathbb{R}^N}\Big(|x|^{-\alpha} \ast U_{\lambda_n,z_n}^{2_{\alpha}^{\ast}}\Big)U_{\lambda_n,z_n}^{2_{\alpha}^{\ast}-2}w_n^2 dx \\
    & + 4c_n^2d_n^2(2_{\alpha}^{\ast})^2|c_n|^{2_{\alpha}^{\ast}(2_{\alpha}^{\ast}-1)}d_n^2 \int_{\mathbb{R}^N}\Big(|x|^{-\alpha} \ast (U_{\lambda_n,z_n}^{2_{\alpha}^{\ast}-1}w_n)\Big)U_{\lambda_n,z_n}^{2_{\alpha}^{\ast}-1}w_n dx
    + o(d_n^2),
    \end{split}
    \end{equation}
    since
    \begin{equation*}
    \begin{split}
    \int_{\mathbb{R}^N}\Big(|x|^{-\alpha} \ast (U_{\lambda_n,z_n}^{2_{\alpha}^{\ast}-1}w_n)\Big)U_{\lambda_n,z_n}^{2_{\alpha}^{\ast}} dx
    = & \int_{\mathbb{R}^N}\Big(|x|^{-\alpha} \ast U_{\lambda_n,z_n}^{2_{\alpha}^{\ast}}\Big)U_{\lambda_n,z_n}^{2_{\alpha}^{\ast}-1}w_n dx \\
    = & \int_{\mathbb{R}^N}\nabla U_{\lambda_n,z_n} \cdot\nabla w_n dx=0,
    \end{split}
    \end{equation*}
    and
    \begin{equation*}
    \int_{\mathbb{R}^N}\Big(|x|^{-\alpha} \ast U_{\lambda_n,z_n}^{2_{\alpha}^{\ast}}\Big)w_n^2 dx
    =\int_{\mathbb{R}^N}\Big(|x|^{-\alpha} \ast (U_{\lambda_n,z_n}^{2_{\alpha}^{\ast}-2}w_n^2)\Big)U_{\lambda_n,z_n}^{U_{\lambda_n,z_n}^{2_{\alpha}^{\ast}}}dx.
    \end{equation*}

    Now, we first prove (\ref{rtnmb}). From (\ref{pfvlnv}) and (\ref{pfvlnv2}) we can obtain
    \begin{equation}\label{pfvlnvy}
    \begin{split}
    \int_{\mathbb{R}^N}\Big(|x|^{-\alpha} \ast (U_{\lambda_n,z_n}^{2_{\alpha}^{\ast}-1}w_n)\Big)U_{\lambda_n,z_n}^{2_{\alpha}^{\ast}-1}w_n dx
    \leq \big\|w_n\big\|^2_{\mathcal{D}^{1,2}(\mathbb{R}^N)}=1,
    \end{split}
    \end{equation}
    and
    \begin{equation}\label{pfvlnv2s}
    \begin{split}
    \int_{\mathbb{R}^N}\Big(|x|^{-\alpha} \ast U_{\lambda_n,z_n}^{2_{\alpha}^{\ast}}\Big)U_{\lambda_n,z_n}^{2_{\alpha}^{\ast}-2}w_n^2 dx
    \leq \big\|w_n\big\|^2_{\mathcal{D}^{1,2}(\mathbb{R}^N)}=1.
    \end{split}
    \end{equation}
    Furthermore, Proposition \ref{propep} implies that
    \begin{equation*}
    \begin{split}
    & \int_{\mathbb{R}^N}|\nabla w_n|^2 dx
    + \int_{\mathbb{R}^N}\Big(|x|^{-\alpha} \ast U_{\lambda_n,z_n}^{2_{\alpha}^{\ast}}\Big) U_{\lambda_n,z_n}^{2_{\alpha}^{\ast}-2}w_n^2 dx \\
    \geq & \mu_{k+N+3}\left[\int_{\mathbb{R}^N}\Big(|x|^{-\alpha} \ast (U_{\lambda_n,z_n}^{2_{\alpha}^{\ast}-1}w_n)\Big)U_{\lambda_n,z_n}^{2_{\alpha}^{\ast}-1}w_n dx
    + \int_{\mathbb{R}^N}\Big(|x|^{-\alpha} \ast U_{\lambda_n,z_n}^{2_{\alpha}^{\ast}}\Big)U_{\lambda_n,z_n}^{2_{\alpha}^{\ast}-2}w_n^2 dx\right],
    \end{split}
    \end{equation*}
    thus
    \begin{equation}\label{epkeyibbb}
    \begin{split}
    1 \geq \mu_{k+N+3}&\int_{\mathbb{R}^N}\Big(|x|^{-\alpha} \ast (U_{\lambda_n,z_n}^{2_{\alpha}^{\ast}-1}w_n)\Big)U_{\lambda_n,z_n}^{2_{\alpha}^{\ast}-1}w_n dx\\&
    + (\mu_{k+N+3}-1)\int_{\mathbb{R}^N}\Big(|x|^{-\alpha} \ast U_{\lambda_n,z_n}^{2_{\alpha}^{\ast}}\Big)U_{\lambda_n,z_n}^{2_{\alpha}^{\ast}-2}w_n^2 dx.
    \end{split}
    \end{equation}
    Then from (\ref{epkeyiybb}), combining with (\ref{pfvlnvy}), (\ref{pfvlnv2s}) and (\ref{epkeyibbb}) we have
    \begin{equation}\label{epkeyiyxbb}
    \begin{split}
    \int_{\mathbb{R}^N}\Big(|x|^{-\alpha} \ast u_n^{2_{\alpha}^{\ast}}\Big)u_n^{2_{\alpha}^{\ast}} dx
    = & |c_n|^{2\cdot2_{\alpha}^{\ast}}\int_{\mathbb{R}^N}\Big(|x|^{-\alpha} \ast U_{\lambda_n,z_n}^{2_{\alpha}^{\ast}}\Big)U_{\lambda_n,z_n}^{2_{\alpha}^{\ast}} dx  + o(d_n^2) \\
    & + 2_{\alpha}^{\ast}(2_{\alpha}^{\ast}-\mu_{k+N+3})|c_n|^{2(2_{\alpha}^{\ast}-1)}d_n^2 \bigg[\int_{\mathbb{R}^N}\Big(|x|^{-\alpha} \ast U_{\lambda_n,z_n}^{2_{\alpha}^{\ast}}\Big)U_{\lambda_n,z_n}^{2_{\alpha}^{\ast}-2}w_n^2 dx \\
    & +\int_{\mathbb{R}^N}\Big(|x|^{-\alpha} \ast (U_{\lambda_n,z_n}^{U_{\lambda_n,z_n}^{2_{\alpha}^{\ast}}-1}w_n)\Big)U_{\lambda_n,z_n}^{U_{\lambda_n,z_n}^{2_{\alpha}^{\ast}}-1}w_n dx\bigg]\\
    & + 2_{\alpha}^{\ast}|c_n|^{2(2_{\alpha}^{\ast}-1)}d_n^2 \bigg[\mu_{k+N+3}\int_{\mathbb{R}^N}\Big(|x|^{-\alpha} \ast (U_{\lambda_n,z_n}^{2_{\alpha}^{\ast}-1}w_n)\Big)U_{\lambda_n,z_n}^{2_{\alpha}^{\ast}-1}w_n dx \\
    & + (\mu_{k+N+3}-1)\int_{\mathbb{R}^N}\Big(|x|^{-\alpha} \ast U_{\lambda_n,z_n}^{2_{\alpha}^{\ast}}\Big)U_{\lambda_n,z_n}^{2_{\alpha}^{\ast}-2}w_n^2 dx\bigg]
    \\
    \leq & |c_n|^{2\cdot2_{\alpha}^{\ast}}\|U\|^2_{\mathcal{D}^{1,2}(\mathbb{R}^N)}
    + 2_{\alpha}^{\ast}|c_n|^{2(2_{\alpha}^{\ast}-1)}d_n^2\big[2(2_{\alpha}^{\ast}-\mu_{k+N+3})+1\big]+o(d_n^2),
    \end{split}
    \end{equation}
    since
    \begin{equation*}
    \begin{split}
    \int_{\mathbb{R}^N}(|x|^{-\alpha} \ast U_{\lambda_n,z_n}^{2_{\alpha}^{\ast}})U_{\lambda_n,z_n}^{2_{\alpha}^{\ast}} dx
    =\int_{\mathbb{R}^N}|\nabla U_{\lambda_n,z_n}|^2 dx
    =\|U\|^2_{\mathcal{D}^{1,2}(\mathbb{R}^N)}.
    \end{split}
    \end{equation*}
    Thus we have
    \begin{equation}\label{epkeyiyxbbf1}
    \begin{split}
    \left(\int_{\mathbb{R}^N}(|x|^{-\alpha} \ast u_n^{2_{\alpha}^{\ast}})u_n^{2_{\alpha}^{\ast}} dx\right)^{\frac{1}{2_{\alpha}^{\ast}}}
    \leq & c_n^2\left(\|U\|^2_{\mathcal{D}^{1,2}(\mathbb{R}^N)}
    + 2_{\alpha}^{\ast}\left[2(2_{\alpha}^{\ast}-\mu_{k+N+3})+1\right]c_n^{-2}d_n^2 \right)^{\frac{1}{2_{\alpha}^{\ast}}}+o(d_n^2) \\
    = & c_n^2\|U\|_{\mathcal{D}^{1,2}(\mathbb{R}^N)}
    + \left[2(2_{\alpha}^{\ast}-\mu_{k+N+3})+1\right]\big\|U\big\|^{\frac{2}{2_{\alpha}^{\ast}}-2}_{\mathcal{D}^{1,2}(\mathbb{R}^N)}d_n^2+o(d_n^2).
    \end{split}
    \end{equation}
    Therefore,
    \begin{equation}\label{epkeyiydzbb}
    \begin{split}
    & \int_{\mathbb{R}^N}|\nabla u_n|^2 dx
    -S_{HLS}\left(\int_{\mathbb{R}^N}(|x|^{-\alpha} \ast u_n^{2_{\alpha}^{\ast}})u_n^{2_{\alpha}^{\ast}} dx\right)^{\frac{1}{2_{\alpha}^{\ast}}} \\
    \geq & d_n^2+c_n^2\|U\|^2_{\mathcal{D}^{1,2}(\mathbb{R}^N)} \\
    & - S_{HLS}\left[c_n^2\|U\|^{\frac{2}{2_{\alpha}^{\ast}}}_{\mathcal{D}^{1,2}(\mathbb{R}^N)}
    + \left[2(2_{\alpha}^{\ast}-\mu_{k+N+3})+1\right]\big\|U\big\|^{\frac{2}{2_{\alpha}^{\ast}}-2}_{\mathcal{D}^{1,2}(\mathbb{R}^N)}d_n^2+o(d_n^2)\right]  \\
    = & d_n^2 \left[1-\left[2(2_{\alpha}^{\ast}-\mu_{k+N+3})+1\right]S_{HLS}\big\|U\big\|^{\frac{2}{2_{\alpha}^{\ast}}-2}_{\mathcal{D}^{1,2}(\mathbb{R}^N)}\right] \\
    & + c_n^2\|U\|^2_{\mathcal{D}^{1,2}(\mathbb{R}^N)}\left(1-S_{HLS}\big\|U\big\|^{\frac{2}{2_{\alpha}^{\ast}}-2}_{\mathcal{D}^{1,2}(\mathbb{R}^N)}\right)+o(d_n^2)  \\
    = & 2\big(\mu_{k+N+3}-2_{\alpha}^{\ast}\big)d_n^2+o(d_n^2),
    \end{split}
    \end{equation}
    for $d_n\to 0$ since $S_{HLS}=\|U\|_{\mathcal{D}^{1,2}(\mathbb{R}^6)}$, which shows that $(\ref{epkeyiydzbb})$ holds for $n$ sufficiently large, then (\ref{rtnmb}) follows immediately.

    Then, let us prove (\ref{rtnml}). We notice that (\ref{epkeyiybb}) also implies
    \begin{equation}\label{epkeyiybbb}
    \begin{split}
    \left(\int_{\mathbb{R}^N}\Big(|x|^{-\alpha} \ast u_n^{2_{\alpha}^{\ast}}\Big)u_n^{2_{\alpha}^{\ast}} dx\right)^{\frac{1}{2_{\alpha}^{\ast}}}
    \geq & \left(|c_n|^{2\cdot2_{\alpha}^{\ast}}\int_{\mathbb{R}^N}\Big(|x|^{-\alpha} \ast U_{\lambda_n,z_n}^{2_{\alpha}^{\ast}}\Big)U_{\lambda_n,z_n}^{2_{\alpha}^{\ast}} dx
    + o(d_n^2) \right)^{\frac{1}{2_{\alpha}^{\ast}}} \\
    = & c_n^2\big\|U\big\|^{\frac{2}{2_{\alpha}^{\ast}}}_{\mathcal{D}^{1,2}(\mathbb{R}^N)}+o(d_n^2),
    \end{split}
    \end{equation}
    therefore,
    \begin{equation}\label{epkeyiydzbl}
    \begin{split}
    & \int_{\mathbb{R}^N}|\nabla u_n|^{2_{\alpha}^{\ast}} dx
    -S_{HLS}\left(\int_{\mathbb{R}^N}\Big(|x|^{-\alpha} \ast |u_
    n|^{2_{\alpha}^{\ast}}\Big)|u_n|^{2_{\alpha}^{\ast}} dx\right)^{\frac{1}{2_{\alpha}^{\ast}}} \\
    \leq & d_n^2+c_n^2\|U\|^2_{\mathcal{D}^{1,2}(\mathbb{R}^N)}
    - S_{HLS}\left(c_n^2\big\|U\big\|_{\mathcal{D}^{1,2}(\mathbb{R}^N)}^{\frac{2}{2_{\alpha}^{\ast}}}+o(d_n^2)\right)  \\
    = & d_n^2+o(d_n^2)+ c_n^2\big\|U\big\|^2_{\mathcal{D}^{1,2}(\mathbb{R}^N)}\left(1- S_{HLS}\big\|U\big\|^{\frac{2}{2_{\alpha}^{\ast}}-2}_{\mathcal{D}^{1,2}(\mathbb{R}^N)}\right)  \\
    = & \big(1+o(1)\big)d_n^2,
    \end{split}
    \end{equation}
    for $d_n\to 0$, which shows that $(\ref{epkeyiydzbl})$ holds for $n$ sufficiently large, then (\ref{rtnml}) follows immediately.
    \end{proof}

    Now, we are ready to prove our main results.

\subsection{Proof of Theorem \ref{thmprt}.} We argue by contradiction. In fact, if the theorem is false then there exists a sequence $\{u_n\}\subset \mathcal{D}^{1,2}(\mathbb{R}^N)\backslash \mathcal{M}$ such that
    \begin{equation}\label{pmrl}
    \frac{\int_{\mathbb{R}^N}|\nabla u_n|^2 dx
    -S_{HLS}\left(\int_{\mathbb{R}^N}\Big(|x|^{-\alpha} \ast |u_n|^{2_{\alpha}^{\ast}}\Big)|u_n|^{2_{\alpha}^{\ast}} dx\right)^{\frac{1}{2_{\alpha}^{\ast}}}}
    {{\rm dist}(u_n,\mathcal{M})^2}
    \to +\infty,
    \end{equation}
    or
    \begin{equation}\label{pmrb}
    \frac{\int_{\mathbb{R}^N}|\nabla u_n|^2 dx
    -S_{HLS}\left(\int_{\mathbb{R}^N}\Big(|x|^{-\alpha} \ast |u_n|^{2_{\alpha}^{\ast}}\Big)|u_n|^{2_{\alpha}^{\ast}} dx\right)^{\frac{1}{2_{\alpha}^{\ast}}}}
    {{\rm dist}(u_n,\mathcal{M})^2}
    \to 0,
    \end{equation}
    as $n\to \infty$. By homogeneity, we can assume that $\|u_n\|_{\mathcal{D}^{1,2}(\mathbb{R}^N)}=1$, and after selecting a subsequence we can assume that ${\rm dist}(u_n,\mathcal{M})\to \xi\in[0,1]$ since ${\rm dist}(u_n,\mathcal{M})=\inf_{c\in\mathbb{R}, \lambda>0, z\in\mathbb{R}^N}\|u_n-cU_{\lambda,z}\|_{\mathcal{D}^{1,2}(\mathbb{R}^N)}\leq \|u_n\|_{\mathcal{D}^{1,2}(\mathbb{R}^N)}$. If (\ref{pmrl}) holds, it must be $\xi=0$, then we have a contradiction by Lemma \ref{lemma:rtnm2b}, thus (\ref{pmrl}) does not hold.  If $\xi=0$, then (\ref{pmrb}) does also not hold by Lemma \ref{lemma:rtnm2b}.

    The only possibility  is that $\xi>0$ and (\ref{pmrb}) holds, that is
    \[{\rm dist}(u_n,\mathcal{M})\to \xi>0\quad \mbox{and}\quad \int_{\mathbb{R}^6}|\nabla u_n|^2 dx
    -S_{HLS}\left(\int_{\mathbb{R}^N}\Big(|x|^{-\alpha} \ast |u_n|^{2_{\alpha}^{\ast}}\Big)|u_n|^{2_{\alpha}^{\ast}} dx\right)^{\frac{1}{2_{\alpha}^{\ast}}}\to 0,\]
    as $n\to \infty$. Then we must have
    \begin{equation}\label{wbsi}
    \left(\int_{\mathbb{R}^N}\Big(|x|^{-\alpha} \ast |u_n|^{2_{\alpha}^{\ast}}\Big)|u_n|^{2_{\alpha}^{\ast}} dx\right)^{\frac{1}{2_{\alpha}^{\ast}}}\to \frac{1}{S_{HLS}},\quad \|u_n\|_{\mathcal{D}^{1,2}(\mathbb{R}^N)}=1.
    \end{equation}
    Combining Lions' concentration-compactness principle \cite{Li85-1,Li85-2} and \cite[Theorem 1.3]{DGY22}, there exist two sequences of numbers $\lambda_n,z_n$ such that
    \begin{equation*}
    \lambda_n^{\frac{N-2}{2}}u_n(\lambda_n (x-z_n))\to U_0\quad \mbox{in}\quad \mathcal{D}^{1,2}(\mathbb{R}^N)\quad \mbox{as}\quad n\to \infty,
    \end{equation*}
    for some $U_0\in\mathcal{M}$, which implies
    \begin{equation*}
    {\rm dist}(u_n,\mathcal{M})={\rm dist}\left(\lambda_n^{\frac{N-2}{2}}u_n(\lambda_n (x-z_n)),\mathcal{M}\right)\to 0 \quad \mbox{as}\quad n\to \infty,
    \end{equation*}
    this is a contradiction.
    \qed

\subsection{Proof of Theorem \ref{thmprtb}.}
    We follow the arguments as in \cite{BE91} and also \cite{WaWi03}. We notice that (\ref{rtbdy46}) is invariant whenever we replace $u$ by $|u|$.
    By the rearrangement inequality we have
    \begin{equation*}
    \big\|\nabla u^*\big\|_{L^2(\Omega)}\leq \big\|\nabla u\big\|_{L^2(\Omega)},\quad \big\|u^*\big\|_{L^{\frac{N}{N-2}}_w(\Omega)}= \big\|u\big\|_{L^{\frac{N}{N-2}}_w(\Omega)},
    \end{equation*}
    and
    \begin{equation*}
    \int_{\Omega}\int_{\Omega}\frac{|u^*(y)|^{2_{\alpha}^{\ast}}|u^*(x)|^{2_{\alpha}^{\ast}}}{|x-y|^\alpha} dydx\geq \int_{\Omega}\int_{\Omega}\frac{|u(y)|^{2_{\alpha}^{\ast}}|u(x)|^{2_{\alpha}^{\ast}}}{|x-y|^\alpha} dydx.
    \end{equation*}
    Here, $u^*$ denotes the symmetric decreasing rearrangement of the nonnegative function $u$ extended to zero outside $\Omega$, see \cite{BrL85,Lieb77,Ta76}. Therefore it is suffices to consider the case in which $\Omega$ is a ball of radius $R$ (chosen to have the same volume as the original domain) and $u$ is symmetric decreasing, that is
    \begin{equation}\label{rtbdlb}
    \int_{B_R(0)}|\nabla u|^2 dx
    -S_{HLS}\left(\int_{B_R(0)}\int_{B_R(0)}\frac{|u(x)|^{2_{\alpha}^{\ast}}|u(y)|^{2_{\alpha}^{\ast}}}{|x-y|^{\alpha}} dxdy\right)^{\frac{1}{2_{\alpha}^{\ast}}}
    \geq B' \big\|u\big\|^2_{L^{\frac{N}{N-2}}_w(B_R(0))},
    \end{equation}
    $\forall u\in \mathfrak{R}_0^{1,2}(B_R(0))$. Here $\mathfrak{R}_0^{1,2}(B_R(0))$ consists of all nonnegative and radial functions in $\mathcal{D}^{1,2}_0(B_R(0))$ with support in the closed ball $\overline{B_R(0)}$. Assume that (\ref{rtbdlb}) is not true, then there exists a sequence $\{u_n\}\subset \mathfrak{R}_0^{1,2}(B_R(0))$ such that
    \begin{equation}\label{rtbdlbc}
    \frac{\int_{B_R(0)}|\nabla u_n|^2 dx
    -S_{HLS}\left(\int_{B_R(0)}\int_{B_R(0)}\frac{|u_n(x)|^{2_{\alpha}^{\ast}}|u_n(y)|^{2_{\alpha}^{\ast}}}{|x-y|^{\alpha}} dxdy\right)^{\frac{1}{2_{\alpha}^{\ast}}}}
    {\big\|u_n\big\|^2_{L^{\frac{N}{N-2}}_w(B_R(0))}}
    \to 0,
    \end{equation}
    as $n\to \infty$. By homogeneity, we can assume that $\|\nabla u_n\|_{L^2(B_R(0))}=1$. Since
    \[\big\|u_n\big\|_{L^{\frac{N}{N-2}}_w(B_R(0))}\leq \big\|u_n\big\|_{L^{\frac{N}{N-2}}(B_R(0))}\leq |B_R(0)|^{1/2}\|u_n\|_{L^{2^{\ast}}(B_R(0))}\leq \mathcal{S}^{-1/2}|B_R(0)|^{1/2},\]
    we must have
    \begin{equation*}
    \left(\int_{B_R(0)}\int_{B_R(0)}\frac{|u_n(x)|^{2_{\alpha}^{\ast}}|u_n(y)|^{2_{\alpha}^{\ast}}}{|x-y|^{\alpha}} dxdy\right)^{\frac{1}{2_{\alpha}^{\ast}}}
    \to \frac{1}{S_{HLS}}.
    \end{equation*}
    From the Lions' concentration-compactness principle \cite{Li85-1,Li85-2} and see also \cite[Theorem 1.3]{DGY22} directly, we can find two sequences $\{c_{n}\}\subset\mathbb{R}$ and $\{\lambda_{n}\}\subset\mathbb{R}^+$ satisfying $c_{n}\to 1$ (up to a constant) and $\lambda_{n}\to +\infty$ as $n\to \infty$ such that
    \[{\rm dist}(u_n,\mathcal{M})=\big\|u_n-c_{n}U_{\lambda_{n},0}\big\|_{\mathcal{D}^{1,2}(\mathbb{R}^N)}\to 0.\]
    Let
    \[a:=U(0)=
    \mathcal{S}^{\frac{(N-\alpha)(2-N)}{4(N-\alpha+2)}}[C(N,\alpha)]^{\frac{2-N}{2(N-\alpha+2)}}[N(N-2)]^{\frac{N-2}{4}}.\]
    Since the support of $u_n$ is contained in $\overline{B_R(0)}$ we obtain
     \begin{equation}\label{rtbdlbcdg}
     \begin{split}
    {\rm dist}(u_n,\mathcal{M})^2
    = & \big\|u_n-c_{n}U_{\lambda_{n},0}\big\|^2_{\mathcal{D}^{1,2}(\mathbb{R}^N)}
    \geq c_{n}^2\int_{|x|\geq R}|\nabla U_{\lambda_{n},0}|^2dx \\
    = & (N-2)^2 a^2\omega_{N-1} c_{n}^2\lambda_n^{N+2}\int^{+\infty}_{R} \frac{r^{N+1}}{(1+\lambda^2_nr^2)^N}dr\\
    = & (N-2)^2 a^2\omega_{N-1} c_{n}^2\int^{+\infty}_{R\lambda_n} \frac{s^{N+1}}{(1+s^2)^N}ds\\
    \geq & C^2_{R}c_{n}^2\lambda_{n}^{-N+2},
    \end{split}
    \end{equation}
    for some constant $C_{R}>0$ as $n$ large enough, where $\omega_{N-1}$ is the volume of $\mathbb{S}^{N-1}$, that is
    \begin{equation*}
     \begin{split}
    {\rm dist}(u_n,\mathcal{M})\geq & C_{R}|c_{n}|\lambda_{n}^{-\frac{N-2}{2}}.
    \end{split}
    \end{equation*}
     Therefore, with
    \[U_{\lambda_{n},0}(R)=a\lambda_n^{\frac{N-2}{2}}(1+\lambda^2_nR^2)^{-\frac{N-2}{2}}:=a_{R,n}\to 0\quad \mbox{as}\quad n\to \infty,\]
    we have
    \begin{equation}\label{rtbdlbcdl}
     \begin{split}
    \big\|u_n\big\|_{L^{\frac{N}{N-2}}_w(B_R(0))}
    \leq & \big\|c_{n}(U_{\lambda_{n},0}-a_{R,n})_+\big\|_{L^{\frac{N}{N-2}}_w(B_R(0))}
    + \big\|u_n-c_{n}(U_{\lambda_{n},0}-a_{R,n})_+\big\|_{L^{\frac{N}{N-2}}_w(B_R(0))} \\
    \leq & |c_{n}|\lambda_{n}^{-\frac{N-2}{2}}\big\|U\big\|_{L^{\frac{N}{N-2}}_w(B_R(0))}
    +C'_{R}\big\|u_n-c_{n}U_{\lambda_{n},0}\big\|_{L^{2^{\ast}}(\mathbb{R}^N)} \\
    \leq & C''_{R}{\rm dist}(u_n,\mathcal{M})
    \end{split}
    \end{equation}
    for some constant $C''_{R}>0$ depending only on $R$ as $n$ large enough. Here $g_+$ denotes the positive part of $g$. Thus combining (\ref{rtbdlbc}) and (\ref{rtbdlbcdl}), Theorem \ref{thmprt} yields a contradiction then (\ref{rtbdlb}) follows.

    Finally, let us consider that (\ref{rtbdy46}) does not hold for the strong norm, that is there is a sequence $\{u_n\}\subset \mathcal{D}_0^{1,2}(\Omega)$ such that
    \begin{equation}\label{rtbdlbcs}
    \frac{\int_{\Omega}|\nabla u_n|^2 dx
    -S_{HLS}\left(\int_{\Omega}\int_{\Omega}\frac{|u_n(x)|^{2_{\alpha}^{\ast}}|u_n(y)|^{2_{\alpha}^{\ast}}}{|x-y|^4} dxdy\right)^{\frac{1}{2_{\alpha}^{\ast}}}}
    {\big\|u_n\big\|^2_{L^{\frac{N}{N-2}}(\Omega)}}
    \to 0,
    \end{equation}
    as $n\to \infty$. Here we consider $\Omega=B_R(0)$ and $\{u_n\}\subset \mathfrak{R}_0^{1,2}(B_R(0))$. Indeed, for general bounded domain $\Omega$ we have $B_{R}(x_0)\subset \Omega\subset B_{R'}(x_0)$ for some $R'\geq R>0$ and $x_0\in\mathbb{R}^6$, then we can consider 
    $u_n\in \mathfrak{R}_0^{1,2}(B_{R}(x_0))$ with $u_n(x)\equiv 0$ if $x\in \mathbb{R}^6\backslash B_{R}(x_0)$. 
    We notice that (\ref{rtbdlbcdg}) also implies
    \begin{equation}\label{rtbdlbcdgs}
     \begin{split}
    {\rm dist}(u_n,\mathcal{M})\leq C^0_{R}c_{n}\lambda_{n}^{-\frac{N-2}{2}}\to 0,
    \end{split}
    \end{equation}
    for some constant $C^0_{R}>0$ as $n$ large enough.
    Furthermore,
    \begin{equation}\label{rtbdlbcdls}
     \begin{split}
    \big\|u_n\big\|_{L^{\frac{N}{N-2}}(B_R(0))}
    \geq & \big\|c_{n}U_{\lambda_{n},0}\big\|_{L^{\frac{N}{N-2}}(B_R(0))}
    - \big\|u_n-c_{n}U_{\lambda_{n},0}\big\|_{L^{\frac{N}{N-2}}(B_R(0))} \\
    \geq & C'_{R}|c_{n}|\lambda_{n}^{-\frac{N-2}{2}}\log (R\lambda_n)
    -C''_{R}\big\|u_n-c_{n}U_{\lambda_{n},0}\big\|_{\mathcal{D}^{1,2}(\mathbb{R}^N)} \\
    \geq & C'_{R}[\log (R\lambda_n)-C'''_{R}]{\rm dist}(u_n,\mathcal{M}).
    \end{split}
    \end{equation}
    Then by Theorem \ref{thmprt},
    \begin{equation}\label{rtbdlbcsl}
    \begin{split}
    & \frac{\int_{B_R(0)}|\nabla u_n|^2 dx
    -S_{HLS}\left(\int_{B_R(0)}\int_{B_R(0)}\frac{|u_n(x)|^{2_{\alpha}^{\ast}}|u_n(y)|^{2_{\alpha}^{\ast}}}{|x-y|^\alpha} dxdy\right)^{\frac{1}{2_{\alpha}^{\ast}}}}
    {\big\|u_n\big\|^2_{L^{\frac{N}{N-2}}(\Omega)}}  \\
    \leq & \frac{\int_{B_R(0)}|\nabla u_n|^2 dx
    -S_{HLS}\left(\int_{B_R(0)}\int_{B_R(0)}\frac{|u_n(x)|^{2_{\alpha}^{\ast}}|u_n(y)|^{2_{\alpha}^{\ast}}}{|x-y|^\alpha} dxdy\right)^{\frac{1}{2_{\alpha}^{\ast}}}}
    {C'_{R}[\log (R\lambda_n)-C'''_{R}]{\rm dist}(u_n,\mathcal{M})} \\
    \leq & \frac{B_2}{C'_{R}[\log (R\lambda_n)-C'''_{R}]}\to 0,
    \end{split}
    \end{equation}
    as $n\to \infty$, thus (\ref{rtbdlbcs}) holds 
    and the proof now is complete.
    \qed
\\

\noindent{Acknowledgement}

Shengbing Deng and Xingliang Tian were supported by National Natural Science Foundation of China (No. 11971392). Minbo Yang and Shunneng Zhao were  supported by National Natural Science Foundation of China (No. 11971436, No. 12011530199) and Natural Science Foundation of Zhejiang Province (No. LZ22A010001, No. LD19A010001).


    \end{document}